\documentclass[11pt, reqno]{amsart}

\usepackage{enumitem}

\usepackage{geometry}
\usepackage{color}
\usepackage{amssymb}
\usepackage{amsmath}
\usepackage{arydshln}
\usepackage{hyperref}
\usepackage{bbm}
\usepackage{mathrsfs}

\usepackage{tikz}
\usetikzlibrary{calc, quotes, angles}

\def\BibTeX{{\rm B\kern-.05em{\sc i\kern-.025em b}\kern-.08em
    T\kern-.1667em\lower.7ex\hbox{E}\kern-.125emX}}

\numberwithin{equation}{section}

\newcommand{\R}{\mathbb{R}}

\newcommand{\N}{\mathbb{N}}

\newcommand{\e}{\varepsilon}

\newcommand{\1}{\mathbbm{1}}

\newtheorem{Theorem}{Theorem}[section]
\newtheorem{Proposition}[Theorem]{Proposition}

\newtheorem{Lemma}[Theorem]{Lemma}
\newtheorem{Remark}[Theorem]{Remark}
\newtheorem{Definition}[Theorem]{Definition}

\begin{document}

\title[On uniqueness and stability for the Enskog equation]{On uniqueness and stability for the Enskog equation}

\author{Martin Friesen}
\address[Martin Friesen]{Faculty of  Mathematics und Natural Sciences, University of Wuppertal, Gau\ss stra\ss e 20, 42119 Wuppertal}
\email[Martin Friesen]{friesen@math.uni-wuppertal.de}

\author{Barbara R\"udiger}
\address[Barbara R\"udiger]{Faculty of  Mathematics und Natural Sciences, University of Wuppertal, Gau\ss stra\ss e 20, 42119 Wuppertal}
\email[Barbara R\"udiger]{ruediger@uni-wuppertal.de}

\author{Padmanabhan Sundar}
\address[Padmanabhan Sundar]{Department of Mathematics, Louisiana State University, Baton Rouge, Louisiana 70803, USA}
\email[Padmanabhan Sundar]{psundar@math.lsu.edu}

\date{\today}

\subjclass[2010]{Primary 35Q20; Secondary 76P05; 76N10; 60H30}

\keywords{Enskog equation; Wasserstein-distance; uniqueness; stability}

\begin{abstract}
The time-evolution of a moderately dense gas in a vacuum is described in classical mechanics by a particle density function obtained from the Enskog equation.
Based on a McKean-Vlasov stochastic equation with jumps, the associated stochastic process was recently studied in \cite{ARS17}. The latter work was extended in \cite{FRS18} to the case of general hard and soft potentials without Grad's angular cut-off assumption.
By the introduction of a shifted distance that exactly compensates for the free transport term that accrues in the spatially inhomogeneous setting,
we prove in this work an inequality on the Wasserstein distance for any two measure-valued solutions to the Enskog equation. As a particular consequence, we find sufficient conditions for the uniqueness and continuous-dependence on initial data for solutions to the Enskog equation applicable to hard and soft potentials without angular cut-off.
\end{abstract}

\maketitle

\allowdisplaybreaks

\section{Introduction}

\subsection{The Boltzmann-Enskog model}
In the classical description of a moderately dense gas in a vacuum, each particle is completely described by its position $r \in \R^d$ and its velocity $v \in \R^d$, where $d \geq 3$. Moreover, the particles are assumed to be indistinguishable and with equal mass.
Any particle $(r,v)$ moves with constant speed $v$ until it performs a collision with another particle $(q,u)$.
Denote by $v^{\star},u^{\star}$ the resulting velocities after collision.
We suppose that collisions are elastic, as a consequence conservation of momentum and kinetic energy hold, i.e.
\begin{align*}
  u + v &= u^{\star} + v^{\star} \\ |u|^2 + |v|^2 &= |u^{\star}|^2 + |v^{\star}|^2.
\end{align*}
A commonly used parameterization of the deflected velocities $v^{\star},u^{\star}$ is given by the unit vector $n = \frac{v^{\star} - v}{|v^{\star} - v|}$ via
\begin{align}\label{PARA}
 \begin{cases} v^{\star} &= v + (u-v,n)n \\ u^{\star} &= u - (u-v,n)n \end{cases}, \ \ n \in S^{d-1},
\end{align}
where $(\cdot,\cdot)$ denotes the euclidean product in $\R^d$.
Note that, for fixed $n \in S^{d-1}$, the change of variables $(v,u) \longmapsto (v^{\star}, u^{\star})$
is an involutive transformation with Jacobian equal to $1$.

Let $f_0(r,v) \geq 0$ be the particle density function of the gas at initial time $t = 0$.
The time evolution $f_t = f_t(r,v)$ is then obtained from the (Boltzmann-)Enskog equation
\begin{align}\label{EQ:03}
 \frac{\partial f_t}{\partial t} + v \cdot (\nabla_r f_t) = \mathcal{Q}(f_t,f_t), \ \ f_t|_{t=0} = f_0, \ \ t > 0.
\end{align}
Here $\mathcal{Q}$ is a non-local, nonlinear collision integral operator given by
\begin{align}\label{CINT}
 \notag &\ \mathcal{Q}(f_t,f_t)(r,v) 
 \\ &= \int_{\R^{2d}}\int_{S^{d-1}}\left( f_t(r,v^{\star})f_t(q,u^{\star}) - f_t(r,v)f_t(q,u)\right)\beta(r-q)B(|v-u|,n)dn du dq,
\end{align}
where $dn$ denotes the Lebesgue surface measure on the sphere $S^{d-1}$ and $B(|v-u|, n) \geq 0$ the collision kernel 
so that $B(|v-u|,n)dn$ includes the effect of velocity cross-section. The particular form of $B(|v-u|,n)$ depends on the particular microscopic model one has in mind,
while $\beta(r-q) \geq 0$ describes the rate at which a particle at position $r$ performs a collision with another particle at position $q$. For instance, $\beta(r-q) = \delta_0(r-q)$ describes the case of local collisions governed by the classical Boltzmann equation, while the particular choice $\beta(r-q) = \delta_{\rho}(|r-q|)$ describes the case where particles behave like billiard balls of radius $\rho > 0$ and was studied by Rezakhanlou \cite{R03}. Following \cite{ARS17} and \cite{FRS18} we study in this work the case where $\beta$ is a symmetric and smooth function. Applications, additional physical background and classical mathematical results are collected in the books of Cercignani \cite{C88} and Cercignani, Illner, Pulvirenti \cite{CIP94}. For recent review articles on this topic we refer to Villani \cite{V02} and Alexandre \cite{A09}.

\subsection{Examples in dimension $d = 3$}
Let us briefly comment on particular examples of collision kernels $B(|v-u|,n)$ in dimension $d = 3$.
Boltzmann's original model was first formulated for (true) hard spheres, i.e. $B(|v-u|,n) = (u-v,n)$. 
A transformation in polar coordinates to a system where the center is in 
$\frac{u+v}{2}$ and $e_3 = (0,0,1)$ is parallel to $u-v$, i.e.
$|u-v|e_3 = u-v$ leads to
\begin{align}\label{boltzmann model}
B(|v-u|,n)dn = |(v-u,n)| = |v-u|\sin\left( \frac{\theta}{2}\right)\cos\left( \frac{\theta}{2}\right)d\theta d\phi,
\end{align}
where $\theta \in (0,\pi]$ is the angle between $u-v$ and $u^{\star} - v^{\star}$ and $\phi \in (0,2\pi]$ is the longitude angle, see Tanaka \cite{T78} or Horowitz and Karandikar \cite{HK90}. This is summarized in Figure 1.
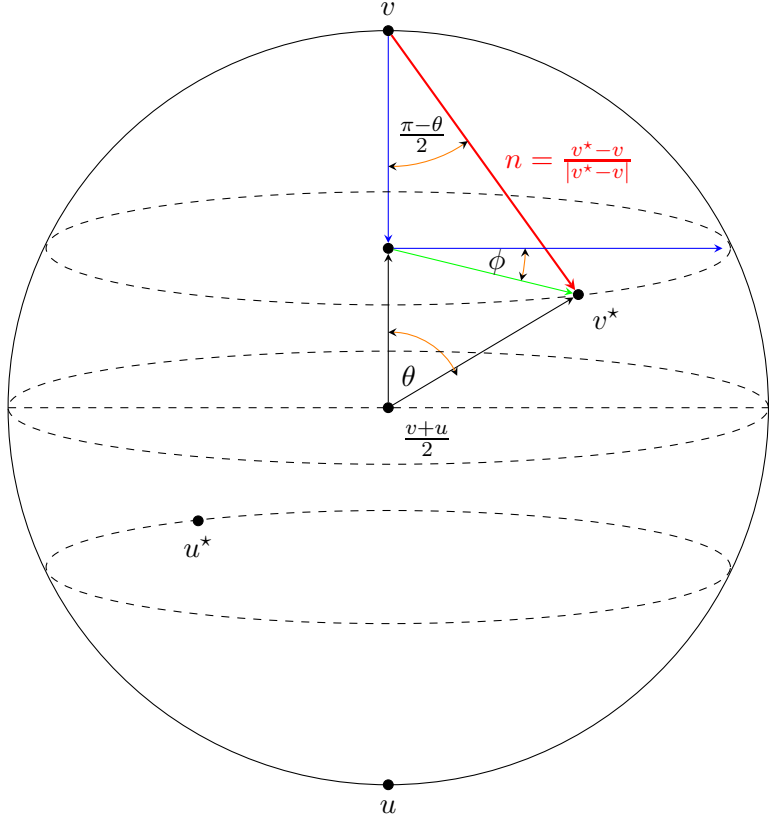
\begin{figure}
\begin{center}
\begin{tikzpicture}
  [
    scale=5,
    >=stealth,
    point/.style = {draw, circle,  fill = black, inner sep = 1.3pt},
    dot/.style   = {draw, circle,  fill = black, inner sep = .2pt},
  ]

  \def \rad{1}
  \node (origin) at (0,0) 
  [point, label = {below right:$\frac{v+u}{2}$}]{};;
  \draw (origin) circle (\rad);

  \node (n1) at +(90:\rad) [point, label = above:$v$] {};
  \node (n2) at +(-90:\rad) [point, label = below:$u$] {};
  
  \node (n3) at +(25:\rad)  {};
  \node (n7) at +(-25:\rad)  {};
  \node (n6) at (0.5,0.3) [point, label = {below right:$v^{\star}$}] {};
  
  \node (n4) at (-0.5, -0.3) [point, label = {below:$u^{\star}$}] {};
  
  \node (n5) at (n3-|0,1) [point] {};
  
  \draw[thick, ->, red] (n1) -- node (a) [label = {right:$n = \frac{v^{\star} - v}{|v^{\star} - v|}$}] {} (n6);
    
  \draw[->, green] (n5) -- (n6);
  \draw[->, blue] (n5) -- (n3);
  \draw[->, blue] (n1) -- (n5);
  \draw 
  pic["$\phi$", draw=orange, <->, angle eccentricity=0.8, angle radius=1.8cm] {angle= n6--n5--n3};
  
  \draw 
  pic["$\frac{\pi - \theta}{2}$", draw=orange, <->, angle eccentricity=0.8, angle radius=1.8cm] {angle= n5--n1--n6};
    
  \draw[->] (origin) -- (n6);
  \draw[->] (origin) -- (n5);
  \draw
   pic["$\theta$", draw=orange, <->, angle eccentricity=0.5, angle radius=1.0cm] {angle= n3--origin--n1};
   
  \draw[dashed] (n3-|0,0) ellipse (0.9 and 0.15);
  \draw[dashed] (n7-|0,0) ellipse (0.9 and 0.15);
  \draw[dashed] (origin) ellipse (1 and 0.15);
  
  \draw[dashed] (-1,0) to (1,0);

\end{tikzpicture}
\end{center}
\caption{Parameterization of collisions}
\end{figure}
Note that Boltzmann's original model \eqref{boltzmann model} 
satisfies Grad's angular cut-off assumption, i.e.
\[
 \int_{S^{2}}B(|v-u|, n) dn < \infty.
\]
A mathematically more challenging class of models
which does not satisfy Grad's angular cut-off assumption
is provided by \textit{long-range interactions} given by
  \begin{align}\label{EQ}
    B(|v-u|, n)dn = |v-u|^{\gamma} b(\theta)d\theta d\xi,
  \end{align}
  where $b$ is at least locally bounded on $(0,\pi]$ and
  \[
    b(\theta) \sim \theta^{-1 - \nu}, \ \ \theta \to 0^+, \  \ \nu \in (0,2).
  \]
  The parameters $\gamma$ and $\nu$ are related by
 \begin{align}\label{INTRO:00}
  \gamma = \frac{s - 5}{s-1}, \qquad \nu = \frac{2}{s-1}, \ \ \ s > 2.
 \end{align}
 For long-range interactions one distinguishes between the following cases:
 \begin{enumerate}
  \item[(i)] Very soft potentials $s \in (2,3]$, $\gamma \in (-3,-1]$ and $\nu \in [1,2)$.
  \item[(ii)] Soft potentials $s \in (3,5)$, $\gamma \in (-1, 0)$ and $\nu \in ( \frac{1}{2}, 1)$.
  \item[(iii)] Maxwellian molecules $s = 5$, $\gamma = 0$ and $\nu = \frac{1}{2}$.
  \item[(iv)] Hard potentials $s > 5$, $\gamma \in (0,1)$ and $\nu \in (0,\frac{1}{2})$.
 \end{enumerate}
  For additional details and comments we refer to \cite{V02} or \cite{A09}.
  Note that one has 
 \[
  \int_{0}^{\pi}b(\theta)d\theta = \infty \ \ \text{ but } \ \ \int_{0}^{\pi}\theta^2 b(\theta)d\theta < \infty.
 \]
  Hence $\mathcal{Q}$ is in this case a non-linear and singular integral operator with either unbounded or singular coefficients.
  A rigorous analysis of the corresponding Cauchy problem \eqref{EQ:03} is therefore a challenging mathematical task.

\section{Statement of the result}

\subsection{Different parameterization of collisions}
In order to study solutions to the Enskog equation it 
is feasible to find continuity properties 
of the deflected velocities $v^{\star}, u^{\star}$ when the incoming velocities $v,u$ are varied. 
Having in mind the case of long-range interactions \eqref{EQ} it is also feasible to parameterize $v^{\star}, u^{\star}$ in terms of the angle $\theta$, i.e. $n = n(v,u,\theta,\phi)$,
and hence study continuity properties of $(u-v,n)n$ in $u,v$ for fixed $\theta,\phi$.
It was already pointed out by Tanaka that in $d = 3$,
$(u,v) \longmapsto (u-v,n)n$ cannot be smooth.
To overcome this problem he introduced in \cite{T78} another transformation of parameters which is bijective, has jacobian 1 and hence can be used on the right side of \eqref{PARA}.
Such ideas have been extended to arbitrary dimension $d \geq 3$
and are briefly summarized in this section, see \cite{FM09} and \cite{LM12}. 
For this purpose set 
$S^{d-2} = \{ \xi \in \R^{d-1} \ | \ |\xi| = 1 \}$ and define
\[
 S^{d-2}(u-v) = \{ \omega \in \R^d \ | \ |u-v| = |\omega|, \ \ (u-v, \omega) = 0\}.
\]
The following is due to \cite{T78} and \cite{FM09}, see also \cite{FRS18} for this formulation.
\begin{Lemma}
 Let $u,v \in \R^d$ with $u \neq v$ and take $n \in S^{d-1}$. Then there exist $(\theta,\xi) \in (0,\pi] \times S^{d-2}$ 
 and a measurable bijective function $\Gamma(u-v, \cdot): S^{d-2} \longrightarrow S^{d-2}(u-v),\ \xi \longmapsto \Gamma(u-v, \xi)$
 such that
\begin{align}\label{N}
 n = \sin\left( \frac{\theta}{2} \right) \frac{u-v}{|u-v|} + \cos\left( \frac{\theta}{2}\right) \frac{\Gamma(u-v,\xi)}{|u-v|},
\end{align}
 where $\theta = \theta(n) \in (0,\pi]$ be the angle between $v^{\star} - u^{\star}$ and $v-u$, 
 i.e. it holds that $(v - u, v^{\star} - u^{\star}) = \cos(\theta) |v-u| | v^{\star} - u^{\star}|$.
\end{Lemma}
The representation of the vector $n$ in \eqref{N} corresponds to the blue lines in Figure 1.
Inserting this into \eqref{PARA} gives after a short computation
\begin{align}\label{PARA:01}
 \begin{cases} v^{\star} &= v + \alpha(v,u,\theta,\xi)
 \\ u^{\star} &= u - \alpha(v,u,\theta,\xi) \end{cases},
\end{align}
where
\begin{align}\label{FPE:00}
 \alpha(v,u,\theta,\xi) = \sin^2\left( \frac{\theta}{2}\right)(u-v) + \frac{\sin(\theta)}{2}\Gamma(u-v,\xi).
\end{align}
Note that \eqref{PARA:01} remains true also for $v = u$, if we let $\alpha(v,v,\theta,\xi) = 0 $, i.e. set $\Gamma(0,\xi) = 0$ in \eqref{FPE:00}. Using \eqref{FPE:00} one finds for all $u,v \in \R^d$, $\theta \in (0,\pi]$ and $\xi \in S^{d-2}$ the identity
 \begin{align}\label{PARA:00}
 |\alpha(v,u,\theta,\xi)| = |v-u| \sin\left( \frac{\theta}{2}\right).
\end{align}
From now on we work with the parameterization \eqref{PARA:01}, where $\alpha$ is given by \eqref{FPE:00}.

\subsection{Some notation}
Here and below we let $\langle v \rangle := (1 + |v|^2)^{1/2}$ and frequently use the elementary inequalities
\[
 \langle v + w\rangle \leq \sqrt{2}\left(\langle v \rangle + \langle w \rangle\right) \qquad \text{ and } \qquad \langle v+w \rangle \leq \sqrt{2}\langle v \rangle \langle w \rangle.
\]
We denote by $K, C > 0$ generic constants which may vary from line to line. Finally, for $k \in \N$ we use the following function spaces
\begin{itemize}
 \item $C^k(\R^{2d})$ the space of all continuous functions on $\R^{2d}$ which are $k$-times continuously differentiable.
 \item $C_b^k(\R^{2d})$ the space of all $f \in C^k(\R^{2d})$ such that $f$ and its first $k$ derivatives are bounded. 
 \item $C_c^k(\R^{2d})$ the space of all $f \in C^k(\R^{2d})$ such that $f$ has compact support.
 \item $\mathrm{Lip}(\R^{2d})$ the space of all globally Lipschitz continuous functions.
\end{itemize}
Denote by $\mathcal{P}(\R^d)$ the space of probability measures and let 
\[
 \langle \psi, \mu \rangle = \int_{\R^{2d}}\psi(r,v)d\mu(r,v)
\]
be the pairing between $\mu \in \mathcal{P}(\R^d)$ and an integrable function $\psi$.

\subsection{Weak formulation for measure-solutions of the Enskog equation}

In this work we take any dimension $d \geq 3$ and assume that the collision kernel $B$ is given by the velocity cross-section $\sigma \geq 0$ and a measure $Q$ such that
\begin{align}\label{BCOLL}
 B(|v-u|,n)dn \equiv \sigma(|v-u|)Q(d\theta)d\xi, \qquad \kappa := \int_{(0,\pi]}\theta Q(d\theta) < \infty
\end{align}
where $d\xi$ is the Lebesgue surface measure on $S^{d-2}$ (recall \eqref{N}).
Moreover suppose that $0 \leq \beta \in C_c^1(\R^{d})$ is symmetric and, there exists $\gamma \in (-d,2]$ and $c_{\sigma} \geq 1$ such that
 \[
  |\sigma(|z|) - \sigma(|w|)| \leq c_{\sigma}| |z|^{\gamma} - |w|^{\gamma}|, \ \ z,w \in \R^d \backslash \{0\}.
 \]
and
\[
 \sigma(|z|) \leq c_{\sigma} \begin{cases} |z|^{\gamma}, & \gamma \in (-d,0] \\ (1 + |z|^2)^{\frac{\gamma}{2}}, & \gamma \in [0,2] \end{cases}.
\]
Without loss of generality we assume that $\beta$ is bounded by $1$.
\begin{Remark}
These conditions are satisfied for $\sigma(z) = |z|^{\gamma}$ and also $\sigma(|z|) = (1+|z|^2)^{\frac{\gamma}{2}}$ with $\gamma \in (-d,2]$.
In particular, we cover the case of hard and soft potentials, provided $s > 3$.
\end{Remark}

Below we describe the weak formulation of the Enskog equation for measures, see \cite{ARS17} and \cite{FRS18} for additional details.
Set $\Xi = (0,\pi] \times S^{d-2}$ and, for $\psi \in C_b^1(\R^d)$, let
\begin{align}\label{ENSKOG:OPERATOR}
 (\mathcal{A}\psi)(r,v;q,u) &= v \cdot(\nabla_r \psi)(r,v) + \sigma(|v-u|)\beta(r-q)(\mathcal{L}\psi)(r,v;u),
 \\ \notag (\mathcal{L}\psi)(r,v;u) &= \int_{\Xi}\left( \psi(r,v + \alpha(v,u,\theta,\xi)) - \psi(r,v)\right) Q(d\theta)d\xi.
\end{align}
By \eqref{PARA:00} we obtain 
$|\alpha(v,u,\theta,\xi)| \leq \theta |v-u|$ and 
\begin{align}\label{FPE:02}
 |\psi(r,v + \alpha(v,u,\theta,\xi)) - \psi(r,v)| \leq \theta |v-u|\max\limits_{|\zeta| \leq 2 (|v| + |u|)} |\nabla_{\zeta}\psi(r,\zeta)|.
\end{align}
In particular, $(\mathcal{L}\psi)(r,v;u)$ is well-defined for all $r,v,u$ and all $\psi \in C^1(\R^{2d})$.
Moreover, if $\psi \in C_b^1(\R^{2d})$, then we obtain
\begin{align}\label{eq: estimate A}
 |\mathcal{A}\psi(r,v;q,u)| \leq \Vert \nabla_r \psi \Vert_{\infty} |v| + \Vert \nabla_v \psi \Vert_{\infty} |v-u|\sigma(|v-u|)|S^{d-2}|\kappa.
\end{align}

\begin{Definition}
 Let $\mu_0 \in \mathcal{P}(\R^{2d})$ and fix $T > 0$.
 A weak solution to the Enskog equation is a family $(\mu_t)_{t \in [0,T]} \subset \mathcal{P}(\R^{2d})$ such that
 \begin{align}\label{INTRO:02}
  \int_{0}^{T}\int_{\R^{4d}}|v-u|\sigma(|v-u|)\mu_t(dr,dv)\mu_t(dq,du)dt < \infty
 \end{align}
 and, for any $\psi \in C_b^1(\R^{2d})$, we have
 \begin{align}\label{FPE:ENSKOG}
  \langle \psi, \mu_t \rangle = \langle \psi, \mu_0 \rangle + \int_0^t \langle \mathcal{A}\psi, \mu_s \otimes \mu_s \rangle ds, \ \ t \in [0,T].
 \end{align}
 A solution is conservative if it has finite second moments in $v$ and
 \[
  \int_{\R^{2d}}\begin{pmatrix} v \\ |v|^2 \end{pmatrix}\mu_t(dr,dv) = \int_{\R^{2d}}\begin{pmatrix} v \\ |v|^2 \end{pmatrix} \mu_0(dr,dv), \ \ t \in [0,T].
 \]
 Analogously we define a global weak solution to the Enskog equation.
\end{Definition}
Note that one has $\mathcal{A}1 = 0$ where $1$ denotes the constant function equal to one. Hence total mass is conserved and we may restrict our study of the Enskog equation without loss of generality to the case of probability distributions.
A construction of such solutions was recently studied in \cite{ARS17} and \cite{FRS18}.

\subsection{Stability estimates for the Enskog equation}

In this work we prove stability estimates 
for weak solutions to the Enskog equation 
in the Wasserstein distance.
In contrast to the space-homogeneous case studied in \cite{FM09}  the additional free transport term $v \cdot \nabla_r$ prevents us from directly applying their methods. 
In order to take this transport of particles into account we introduce the \textit{shifted} Wasserstein distance
\[
 W_1^t(\mu,\nu) := W_1(S(-t)^*\mu, S(-t)^*\nu), \qquad t \in \R,
\]
where $S(t)$ is a one-parameter group of transformations defined by
\[
 S(t)\psi(r,v) = \psi(r+tv,v), \qquad (r,v) \in \R^{2d}, \ \ t \in \R,
\] 
$S(t)^*$ denotes the adjoint operator to $S(t)$ 
acting on measures $\mu \in \mathcal{P}(\R^{2d})$ and $\psi \in C_b(\R^{2d})$ via
\begin{align}\label{EQ:06}
 \langle S(t)\psi, \mu \rangle = \langle \psi, S(t)^*\mu \rangle, \qquad t \in \R.
\end{align}
and $W_1$ denotes the classical Wasserstein distance, i.e.
\[
 W_1(\mu, \nu) := \sup_{\Vert \psi \Vert_{\mathrm{Lip}} \leq 1} \langle \psi, \mu - \nu \rangle,
\]
where $\mathrm{Lip}(\R^{2d}) = \{ \psi \ | \ \| \psi \|_{\mathrm{Lip}} < \infty\}$ is the space of all globally Lipschitz continuous functions and
\[
 \| \psi \|_{\mathrm{Lip}} = \sup \limits_{(r,v) \neq (q,u)} \frac{|\psi(r,v) - \psi(q,u)|}{|r-q| + |v-u|}.
\]
In the case of hard potentials we obtain the following.
\begin{Theorem}\label{FPE:UNIQTH:00}
 Suppose that $\gamma \in [0,2]$, fix $T > 0$ and $\delta > 0$. Then there exists a constant $K > 0$ 
 such that for two given weak solutions $(\mu_t)_{t \in [0,T]}, (\nu_t)_{t \in [0,T]}$ to the Enskog equation 
 satisfying 
 \begin{align}\label{EQ:07}
  \mathcal{C}_{\gamma}(T,\mu + \nu, \delta) := \sup \limits_{t \in [0,T]}\int_{\R^{2d}}\left( e^{\delta |v|^{1+\gamma}} + |r|^{1+\delta} \right)(\mu_t + \nu_t)(dr,dv) < \infty
 \end{align}
 we have, for $t \in [0,T]$,
 \begin{align*}
   W_1^t(\mu_t,\nu_t) &\leq W_1(\mu_0, \nu_0) 
   \\ & \ \ \ + K \mathcal{C}_{\gamma}(T,\mu+ \nu,\delta) 
   \int_0^t W^s_1(\mu_s,\nu_s)\left(1  + |\log(W_1^s(\mu_s,\nu_s))|\right)ds.
 \end{align*}
\end{Theorem}
For soft potentials we obtain the following.
\begin{Theorem}\label{FPE:UNIQTH:01}
 Suppose that $\gamma \in (-d,0)$ and fix $T > 0$.
 \begin{enumerate}
  \item[(a)] If $\gamma \in (-d,-1]$, then there exists a constant $K > 0$ such that for two given weak solutions $(\mu_t)_{t \in [0,T]}, (\nu_t)_{t \in [0,T]}$ to the Enskog equation satisfying
  \begin{align}\label{EQ:08}
   \Lambda(\mu_t + \nu_t):= \sup \limits_{u \in \R^d} \int_{\R^{2d}}|v-u|^{\gamma}(\mu_t + \nu_t)(dr,dv) < \infty
  \end{align}
  and \eqref{EQ:07} we have, for $t \in [0,T]$,
  \[
   W_1^t(\mu_t,\nu_t) \leq W_1(\mu_0,\nu_0)\exp\left( K\int_{0}^{t}(1 + \Lambda(\mu_s + \nu_s))ds\right).
  \]
  \item[(b)] If $\gamma \in (-1,0)$. Then for each $\delta > 0$ there exists a constant $K > 0$ such that 
  for two given weak solutions $(\mu_t)_{t \in [0,T]}, (\nu_t)_{t \in [0,T]}$ to the Enskog equation
  satisfying \eqref{EQ:07} and \eqref{EQ:08} we have, for $t \in [0,T]$,
 \begin{align*}
   W_1^t(\mu_t,\nu_t) &\leq W_1(\mu_0, \nu_0) 
  + K \mathcal{C}_{\gamma}(T,\mu+ \nu,\delta)
  \\ &\ \ \ \cdot \sup \limits_{s \in [0,T]} \Lambda(\mu_s + \nu_s) \int_0^t W_1^s(\mu_s,\nu_s)\left(1 + |\log(W_1^s(\mu_s,\nu_s))|\right)ds.
 \end{align*}
 \end{enumerate}
\end{Theorem}
Condition \eqref{EQ:08} stems from the necessity to compensate the singularity of $\sigma(|v-u|)$ at zero
appearing in the case of soft potentials, see \eqref{EQ}.
\begin{Remark}
 Suppose that $\gamma \in (-d,0)$ and let $\mu_t(dr,dv) = f_t(r,v)drdv$, $\nu_t(dr,dv) = g_t(r,v)drdv$.
 Then for each $p > \frac{d}{d+\gamma}$ there exists a constant $C(p,\gamma) > 0$ such that 
 \begin{align*}
  \Lambda(\mu_t + \nu_t) &\leq 2 + C(p,\gamma)
   \bigg[  \left(\int_{\R^d}\left( \int_{\R^d} f_t(r,v)dr \right)^p dv \right)^{1/p} 
   \\ &\qquad \qquad \qquad \qquad +  \left( \int_{\R^d} \left( \int_{\R^d}g_t(r,v)dr \right)^p dv \right)^{1/p} \bigg].
 \end{align*}
\end{Remark}
Above estimates are sufficient to imply uniqueness and stability (with respect to initial data) of weak solutions to the Enskog equation. 
Indeed, by using a generalization of the Gronwall inequality as stated in the appendix (see e.g. \cite[Lemma 5.2.1, p. 89]{C95}) we obtain the following.
\begin{Theorem}
 Fix $T > 0$.
 \begin{enumerate}[leftmargin = *]
  \item[(a)] Let $(\mu_t)_{t \in [0,T]}$, $(\nu_t)_{t \in [0,T]}$ be two weak solutions to the Enskog equation. 
  Suppose one of the following conditions is satisfied:
  \begin{itemize}
   \item $\gamma \in [0,2]$ and there exists $\delta > 0$ with 
   \[
    \mathcal{C}_{\gamma}(T,\mu + \nu, \delta) < \infty.
   \]
   \item $\gamma \in (-1,0)$ and there exists $\delta > 0$ with 
   \[
    \mathcal{C}_{\gamma}(T,\mu + \nu, \delta) + \sup\limits_{t \in [0,T]}\Lambda(\mu_t + \nu_t) < \infty.
   \]
   \item $\gamma \in (-d,-1]$, 
   \[
    \sup \limits_{t \in [0,T]} \left\{ \Lambda(\mu_t + \nu_t) 
    + \int_{\R^{2d}}|v|^2 (\mu_t(dr,dv) + \nu_t(dr,dv)) \right\} < \infty.
   \]
  \end{itemize}
  If $\mu_0 = \nu_0$, then $\mu_t = \nu_t$ for all $t \in [0,T]$.
  \item[(b)] Let $(\mu_t^{(n)})_{t \in [0,T]}$ and $(\mu_t)_{t \in [0,T]}$ be weak solutions to the Enskog equation. 
  Suppose one of the following conditions is satisfied:
  \begin{itemize}
   \item $\gamma \in [0,2]$ and there exists $\delta > 0$ 
   with 
   \[
    \sup\limits_{n \in \N}\mathcal{C}_{\gamma}(T,\mu^{(n)} + \mu, \delta) < \infty.
   \]
   \item $\gamma \in (-1,0)$ and there exists $\delta > 0$ with 
   \[
    \sup \limits_{n \in \N}\left\{ \mathcal{C}_{\gamma}(T,\mu^{(n)} + \mu, \delta) + \sup\limits_{t \in [0,T]}\Lambda(\mu^{(n)}_t + \mu_t) \right\} < \infty.
   \]
   \item $\gamma \in (-d,-1]$ and
   \[
    \sup \limits_{n \in \N}\sup\limits_{t \in [0,T]}\left\{ \Lambda(\mu_t^{(n)} + \mu_t) + \int_{\R^{2d}}|v|^2 (\mu_t^{(n)}(dr,dv) + \mu_t(dr,dv)) \right\} < \infty.
   \]
  \end{itemize}
  If $W_1(\mu_0^{(n)}, \mu_0) \longrightarrow 0$ as $n \to \infty$, then
  \[
    \lim \limits_{n \to \infty}\sup_{t \in [0,T]} W_1^t(\mu_t^{(n)}, \mu_t) = 0.
  \]
 \end{enumerate}
\end{Theorem}
Our proofs are partially inspired by the work of Fournier and Mouhot \cite{FM09} where 
similar estimates for solutions to the space-homogeneous Boltzmann equation have been established.
However, since we work in the space-inhomogeneous setting we have to replace the classical Wasserstein distance $W_1$ to by a shifted distance $W_1^t$ which compensates the free transport operator $v \cdot \nabla_r$ appearing in the definition of $\mathcal{A}$, see \eqref{ENSKOG:OPERATOR}. For hard-potentials the authors have used in \cite{FM09} Povzner inequalities to prove creation of exponential moments for solutions to the (space-homogeneous) Boltzmann equation, see also \cite{LM15} and the references therein. Their proofs implicitly use the fact that any two particles may perform a collision. In contrast to that, in the space-inhomogeneous setting studied here $\beta$ is compactly supported and hence only particles being close enough may perform a collision. This prevents us from proving similar results on the creation of moments for the Enskog equation with hard potentials.

Other uniqueness results for the space-homogeneous Boltzmann equation are based on additional regularity assumptions for the solution, see e.g. \cite{DM09} and \cite{X16}.

\section{Mild formulation for the Enskog equation}
In order to prove the desired stability estimates for the shifted distance $W_1^t$, it is reasonable to use another formulation of the Enskog equation which involves the semigroup $S(t)$. This is precisely the content of this section.
Define for $(r,v), (q,u) \in \R^{2d}$ and $\psi \in \mathrm{Lip}(\R^{2d})$
\[
 (\mathcal{B}\psi)(r,v;q,u) = \sigma(|v-u|)\beta(r-q)(\mathcal{L}\psi)(r,v;u).
\]
Then there exists a constant $C > 0$ such that
for each $\psi \in \mathrm{Lip}(\R^{2d})$ one has
\begin{align}\label{EQ:04}
 |(\mathcal{B}\psi)(r,v;q,u)| \leq C |v - u| \sigma(|v-u|) \| \psi \|_{\mathrm{Lip}}.
\end{align}
The next result is crucial for estimating weak solutions to the Enskog equation.
\begin{Proposition}\label{FPEUNIQ:LEMMA00}
 Fix $T > 0$ and let 
 $(\mu_t)_{t \in [0,T]} \subset \mathcal{P}(\R^{2d})$ satisfy 
 \begin{align}\label{EQ:02}
  \sup \limits_{t \in [0,T]} \int_{\R^{2d}}|v|^2 \mu_t(dr,dv) < \infty.
 \end{align}
 If $(\mu_t)_{t \in [0,T]}$ is a weak solution to the Enskog equation,
 then 
 \begin{align}\label{FPEUNIQ:EQ00}
  \langle \psi, \mu_t \rangle = \langle S(t)\psi, \mu_0 \rangle + \int_{0}^{t}\langle \mathcal{B}S(t-s)\psi, \mu_s \otimes \mu_s \rangle ds, \ \ t \in [0,T].
 \end{align}
 holds for each $\psi \in \mathrm{Lip}(\R^{2d})$.
\end{Proposition}
\begin{proof}
  Fix $\psi \in C_b^2(\R^{2d})$ and $t \in (0,T]$.
 Let us show that the function $[0,t] \ni s \longmapsto \langle S(t-s)\psi,\mu_s\rangle$ is absolutely continuous and for a.a. $s \in [0,t)$ it holds that
 \begin{align}\label{FPEUNIQ:EQ01}
  \frac{d}{ds}\langle S(t-s)\psi, \mu_s\rangle = \langle \mathcal{B}S(t-s)\psi, \mu_s \otimes \mu_s \rangle.
 \end{align}
 In such a case, using \eqref{EQ:02} and \eqref{EQ:04},
 we may integrate \eqref{FPEUNIQ:EQ01} over $[0,t]$ which would readily yield \eqref{FPEUNIQ:EQ00} for $\psi \in C_b^2(\R^{2d})$.
 If $\psi \in \mathrm{Lip}(\R^{2d})$ then we may find a sequence of functions $\psi_n \in C_b^2(\R^{2d})$ such that
 $\sup_{n \in \N} \| \psi_n \|_{\mathrm{Lip}} < \infty$
 and $\psi_n \longrightarrow \psi$ pointwise.
 Hence passing to the limit $n \to \infty$ proves that \eqref{FPEUNIQ:EQ00} also holds for $\psi \in \mathrm{Lip}(\R^{2d})$. 
 
 Arguing in this way, it remains to prove \eqref{FPEUNIQ:EQ01} for each $\psi \in C_b^2(\R^{2d})$ and fixed $t > 0$.
 Take $s \in [0,t)$ and let $h \in \R$ with $|h| \leq (t-s) \wedge 1$.
 Write 
 \begin{align*}
  &\ \frac{ \langle S(t - (s+h))\psi, \mu_{s+h} \rangle - \langle S(t-s)\psi, \mu_s \rangle}{h}
  \\ &= \left \langle \frac{S(t - (s+h))\psi - S(t-s)\psi}{h}, \mu_{s+h}\right \rangle
  + \frac{ \langle S(t-s)\psi, \mu_{s+h} \rangle - \langle S(t-s)\psi, \mu_s\rangle}{h}.
 \end{align*}
 Since $(\mu_t)_{t \geq 0}$ satisfies the Enskog equation
 and $S(t-s)\psi \in C_b^2(\R^{2d})$ we conclude that
 \begin{align}\label{EQ:09}
  \frac{ \langle S(t-s)\psi, \mu_{s+h} \rangle - \langle S(t-s)\psi, \mu_s\rangle}{h} \longrightarrow \langle \mathcal{A}S(t-s)\psi, \mu_s \otimes \mu_s \rangle, \qquad h \to 0
 \end{align}
 for a.a. $s \in [0,t)$. Next we will prove that
 \begin{align}\label{EQ:10}
  \left \langle \frac{S(t - (s+h))\psi - S(t-s)\psi}{h}, \mu_{s+h}\right \rangle \longrightarrow - \int_{\R^{2d}} v \cdot (\nabla_r S(t-s)\psi)(r,v) \mu_s(dr,dv)
 \end{align}
 as $h \to 0$. Combining then \eqref{EQ:09}, \eqref{EQ:10} and using the definition of $\mathcal{A}$ and $\mathcal{B}$ proves \eqref{FPEUNIQ:EQ01}.
 In order to prove \eqref{EQ:10} we let 
 \[
  f_h(r,v) = S(t - (s+h))\psi(r,v) = \psi( r + ( t - s - h)v,v)
 \]
 and denote the corresponding pointwise derivative with respect to $h$ at $h = 0$ by
 \begin{align*}
  f'_0(r,v) &= - v \cdot (\nabla_r \psi)(r + (t-s)v,v).
 \end{align*}
 For $R > 0$ take a smooth function $\varphi_R$ on $\R^d$ such that
 $\1_{ [0,R]}(|v|) \leq \varphi_R(v) \leq \1_{[0,2R]}(|v|)$.
 Then 
  \begin{align*}
  &\ \left \langle \frac{S(t - (s+h))\psi - S(t-s)\psi}{h}, \mu_{s+h}\right \rangle 
  \\ &= \left \langle \frac{f_h - f_0}{h} - f_0', \mu_{s+h} \right \rangle + \langle f_0' , \mu_{s+h} \rangle
  \\ &= \left \langle \frac{f_h - f_0}{h} - f_0', \mu_{s+h} \right \rangle + \langle f_0' \varphi_R , \mu_{s+h} - \mu_s \rangle
  + \langle f_0' (1 - \varphi_R), \mu_{s+h} - \mu_s \rangle
  + \langle f_0', \mu_s \rangle
  \\ &= J_1 + J_2 + J_3 + J_4.
 \end{align*}
 Using the estimate
 \begin{align*}
  \left| \frac{f_h(r,v) - f_0(r,v)}{h} - f_0'(r,v)\right| \leq \frac{|h|}{2} |v|^2 \| \psi \|_{C_b^2}
 \end{align*}
 we obtain
 \begin{align*}
  |J_1| &\leq \frac{|h|}{2}\| \psi\|_{C_b^2} \sup \limits_{|h| \leq (t-s)\wedge 1} \int_{\R^{2d}} |v|^2 \mu_{s+h}(dr,dv).
 \end{align*}
 Similarly, using $|f_0'(r,v)| \leq |v| \| \psi \|_{C_b^1}$
 gives $f_0' \varphi_R \in C_b(\R^{2d})$ and hence we obtain
 \[
  \lim_{h \to 0}J_2 = 0, \qquad \forall R > 0 and s > 0,
 \]
 since $\R_+ \ni s \longmapsto \langle f_0' \varphi_R, \mu_{s} \rangle$ is continuous which essentially follows from \eqref{eq: estimate A} combined with \eqref{FPE:ENSKOG}, and a standard approximation of $C_b^1(\R^{2d})$ functions by $C_b(\R^{2d})$ functions. For $J_3$ we use
 $1 - \varphi_R(v) \leq \1_{(R,\infty)}(|v|)$
 so that
 \begin{align*}
  |J_3| &\leq \| \psi \|_{C_b^1} \int_{|v| > R}|v| ( \mu_{s+h}(dr,dv) + \mu_s( dr,dv))
  \\ &\leq \frac{2 \| \psi \|_{C_b^1} }{R} \sup \limits_{|h| \leq (t-s) \wedge 1} \int_{\R^{2d}}|v|^2 \mu_{s+h}(dr,dv).
 \end{align*}
 Combining these estimates yields \eqref{EQ:10} as $h \to 0$ and letting $R \to \infty$ thus completes the proof of Proposition \ref{FPEUNIQ:LEMMA00}.
\end{proof}

\section{The coupling inequality}

Below we first provide another representation for the metric $W_1^t(\mu, \nu)$ in terms of optimal couplings.
Additional details on the classical Wasserstein distance and optimal transport are given in \cite{V09}.
Let $\mu,\nu \in \mathcal{P}(\R^{2d})$. A coupling $H$ of $(\mu,\nu)$ is a probability measure on $\R^{4d}$ such that
its marginals are given by $\mu$ and $\nu$, respectively.
Let $\mathcal{H}(\mu,\nu)$ the space of all such couplings.
Define a one-parameter family of norms on $\R^{2d}$ via
\[
 |(r,v) - (\widetilde{r},\widetilde{v})|_t := |(r-vt) - (\widetilde{r} - \widetilde{v}t)| + |v- \widetilde{v}|, \ \ \ t \geq 0.
\]
Related to this family we define the Lipschitz norms 
\[
 \| \psi \|_{t} := \sup_{(r,v) \neq (\widetilde{r}, \widetilde{v})} \frac{|\psi(r,v) - \psi(\widetilde{r}, \widetilde{v})|}{|(r,v) - (\widetilde{r},\widetilde{v})|_t}, \ \ \ t \geq 0.
\]
Note that these Lipschitz norms are all equivalent.
We will use the following simple observation.
\begin{Lemma}
 Let $\mu,\nu$ be probability measures with finite first moments and take $t \geq 0$. Then there exists $H_t \in \mathcal{H}(\mu, \nu)$ such that
 \begin{align}\label{COUP:00}
  W_1^t(\mu, \nu) &= \sup\limits_{\| \psi \|_{0} \leq 1} \langle S(-t)\psi, \mu - \nu \rangle 
  \\ \notag &= \sup\limits_{\| \psi \|_{t} \leq 1} \langle \psi, \mu - \nu \rangle
  = \int_{\R^{4d}}|(r,v) - (\widetilde{r}, \widetilde{v})|_t dH_t(r,v;\widetilde{r},\widetilde{v}).
 \end{align}
\end{Lemma}
\begin{proof}
 The first equality in \eqref{COUP:00} follows by definition of $W_1^t$ combined with \eqref{EQ:06}. 
 For the second equality in \eqref{COUP:00} observe that, for any $\psi$ with $\| \psi \|_{0} \leq 1$, we get $\| S(-t)\psi \|_{t} \leq 1$.
 Conversely any  $\psi$ satisfying $\| \psi \|_{t} \leq 1$ 
 can be written as $\psi = S(-t)\phi$ where $\phi := S(t)\psi$ satisfies $\| \phi \|_{0} \leq 1$.
 The last equality in \eqref{COUP:00} is a particular case of the Kantorovich duality for Wasserstein distances (see e.g. \cite[Theorem 5.10]{V09}).
\end{proof}
The following is an extension of \cite[Theorem 2.2]{FM09} to the space-inhomogeneous case.
\begin{Proposition}\label{general coupling inequality}
 Take $T > 0$ and let $(\mu_t)_{t \in [0,T]}, (\nu_t)_{t \in [0,T]}$ be two weak solutions to the Enskog equation satisfying
 \[
  \sup_{t \in [0,T]}\int_{\R^{2d}}|v|^2 ( \mu_t(dr,dv) + \nu_t(dr,dv)) < \infty,
 \]
 and 
 \begin{align}\label{FPE:12}
  &\ \int_{0}^{T} \int_{\R^{4d}}(|v| + |u| + |r| + |q|)\sigma(|v-u|) 
  \\ \notag  & \qquad \qquad \qquad \left( \mu_t(dq,du)\mu_t(dr,dv) + \nu_t(dq,du)\nu_t(dr,dv)\right)dt < \infty.
 \end{align}
 For $t \in [0,T]$ let $H_t \in \mathcal{H}(\mu_t, \nu_t)$ be such that
 \begin{align}\label{FPEUNIQ:EQ02}
  W_1^t(\mu_t,\nu_t) = \int_{\R^{4d}}|(r,v) - (\widetilde{r},\widetilde{v})|_t dH_t(r,v;\widetilde{r},\widetilde{v}).
 \end{align}
 Then for all $t \in [0,T]$
 \begin{align*}
  W_1^t(\mu_t,\nu_t) &\leq W_1(\mu_0,\nu_0) + 2\kappa(1+T)\int_{0}^{t} \int_{\R^{8d}} \Psi dH_s(q,u;\widetilde{q},\widetilde{u})dH_s(r,v;\widetilde{r},\widetilde{v})ds
 \end{align*}
 where
 \begin{align*}
  \Psi &:= \left( |v-u| + |\widetilde{v} - \widetilde{u}| \right)\left| \sigma(|v-u|) \beta(r-q) - \sigma(|\widetilde{v}-\widetilde{u}|)\beta(\widetilde{r}- \widetilde{q})\right|
  \\ &+ \left( |v- \widetilde{v}| + |u - \widetilde{u}|\right) \min\{\sigma(|v-u|)\beta(r-q), \sigma(|\widetilde{v}-\widetilde{u}|)\beta(\widetilde{r}- \widetilde{q})\}.
 \end{align*}
\end{Proposition}

In order to prove this result we need continuity properties for the collision integral, 
i.e. to compare \eqref{FPE:00} for different values of $u,v$. It was already pointed out by Tanaka that $(u,v) \longmapsto \alpha(v,u,\theta,\xi)$ cannot be smooth for any choice of $(\theta,\xi)$. Using th parameterization \eqref{PARA:01}
Tanaka \cite[Lemma 3.1]{T78} has shown that if we allow to shift the angles $\xi$ in a suitable way, then a weaker form of continuity holds. The latter estimate is sufficient for this work.
Below we recall Tanaka's result for arbitrary dimension $d \geq 3$ which is due to \cite{FM09}.
\begin{Lemma}\cite[Lemma 3.1]{FM09}\label{PARAMETARIZATION}
 There exists a measurable map $\xi_0: \R^d \times \R^d \times S^{d-2} \longrightarrow S^{d-2}$ such that for any $X,Y \in \R^d\backslash \{0\}$,
 the map $\xi \longmapsto \xi_0(X,Y,\xi)$ is a bijection with jacobian $1$ from $S^{d-2}$ onto itself, and
 \[
  |\Gamma(X, \xi) - \Gamma(Y, \xi_0(X,Y,\xi))| \leq 3 |X-Y|, \qquad  \xi \in S^{d-2}.
 \]
\end{Lemma}
 With this parameterization we obtain from Lemma \ref{PARAMETARIZATION}, for all $u,v, \widetilde{u}, \widetilde{v} \in \R^d$, all $\theta \in [0,\pi]$ and all $\xi \in S^{d-2}$, we have the inequality
 \begin{align}\label{EQ:13}
  |\alpha(v,u,\theta, \xi) - \alpha(\widetilde{v}, \widetilde{u}, \theta, \xi_0(v-u, \widetilde{v} - \widetilde{u}, \xi))| \leq 2 \theta\left( |v - \widetilde{v}| + |u -  \widetilde{u}|\right).
 \end{align}
 We are now prepared to prove our main coupling inequality of this section.
\begin{proof}[Proof of Proposition \ref{general coupling inequality}]
 Take $\psi \in \mathrm{Lip}(\R^{2d})$ with $\Vert \psi \Vert_{0} \leq 1$. By Proposition \ref{FPEUNIQ:LEMMA00}
 $(\mu_t)_{t \in [0,T]}$ and $(\nu_t)_{t \in [0,T]}$ also satisfy \eqref{FPEUNIQ:EQ00}.
 To shorten notation we let $\alpha = \alpha(v,u,\theta,\xi)$, $\sigma = \sigma(|v-u|)$, $\beta = \beta(r-q)$ 
 and likewise $\widetilde{\alpha}, \widetilde{\sigma}, \widetilde{\beta}$ with $(r,v),(q,u)$ replaced by $(\widetilde{r},\widetilde{v}), (\widetilde{q},\widetilde{u})$.
 Moreover, let $dH_s^0 = dH_s(q,u;\widetilde{q},\widetilde{u})$
 and $dH_s^1 = dH_s(r,v;\widetilde{r},\widetilde{v})$. 
 Using \eqref{FPEUNIQ:EQ00}, $H_s \in \mathcal{H}(\mu_s,\nu_s)$, 
 the definition of $\mathcal{B}$ and finally
 $x = x \wedge y + (x-y)_+$, for $x,y \geq 0$ with $x_+ := \max\{x, 0\}$, we obtain
 \begin{align*}
  &\ \langle S(-t)\psi, \mu_t - \nu_t \rangle - \langle \psi, \mu_{0} - \nu_{0} \rangle 
  \\ &= \int_{0}^{t}\bigg\{\langle \mathcal{B}S(-s)\psi, \mu_s \otimes \mu_s \rangle - \langle \mathcal{B}S(-s)\psi, \nu_s \otimes \nu_s \rangle \bigg\} ds
  \\ &= \int_{0}^{t}\int_{\R^{8d}}\bigg \{ (\mathcal{B}S(-s)\psi)(r,v;q,u) 
  \\ &\hskip30mm - (\mathcal{B}S(-s)\psi)(\widetilde{r},\widetilde{v};\widetilde{q},\widetilde{u}) \bigg \}dH_s(q,u,\widetilde{q},\widetilde{u})dH_s(r,v;\widetilde{r},\widetilde{v}) ds.
  \\ &= \int_{0}^{t}\int_{\R^{8d} \times \Xi} \bigg\{ \left( S(-s)\psi(r,v+\alpha) - S(-s)\psi(r,v)\right)\sigma \beta
  \\ &\hskip30mm - \left( S(-s)\psi(\widetilde{r}, \widetilde{v} + \widetilde{\alpha}) - S(-s)\psi(\widetilde{r},\widetilde{v})\right)\widetilde{\sigma}\widetilde{\beta}\bigg\} dQd\xi dH_s^0 dH_s^1 ds
  \\ &= \int_{0}^{t}\int_{\R^{8d}\times \Xi}\left(\sigma \beta \wedge \widetilde{\sigma}\widetilde{\beta}\right) S(-s)\bigg\{ \psi(r,v+\alpha) - \psi(\widetilde{r}, \widetilde{v} + \widetilde{\alpha}) 
  \\ &\hskip60mm - \psi(r,v) + \psi(\widetilde{r},\widetilde{v})\bigg\} dQd\xi dH_s^0dH_s^1 ds
  \\ &\ \ \ + \int_{0}^{t}\int_{\R^{8d}\times \Xi}\left( \sigma \beta - \widetilde{\sigma}\widetilde{\beta}\right)_+ \left( S(-s)\psi(r,v+\alpha) - S(-s)\psi(r,v)\right)dQd\xi dH_s^0 dH_s^1 ds
  \\ &\ \ \ - \int_{0}^{t}\int_{\R^{8d}\times \Xi}\left( \widetilde{\sigma}\widetilde{\beta} - \sigma\beta \right)_+ \left( S(-s)\psi(\widetilde{r}, \widetilde{v} + \widetilde{\alpha}) - S(-s)\psi(\widetilde{r}, \widetilde{v}) \right) dQd\xi dH_s^0 dH_s^1 ds
  \\ &=: J_1 + J_2 + J_3.
 \end{align*}
 Note that, by \eqref{PARA:00},
  we have $|(0,\alpha)|_s \leq (1+s)|\alpha| \leq (1+T)\theta |v-u|$
  where $0$ denotes the zero vector in $\R^d$.
  Analogously we obtain
 $|(0,\widetilde{\alpha})|_s \leq (1+T)\theta |\widetilde{v} - \widetilde{u}|$. 
 Hence using $\Vert S(-s)\psi \Vert_{s} \leq 1$ we get
 \begin{align*}
  J_2 + J_3 &\leq \int_{0}^{t} \int_{\R^{8d}}\int_{\Xi}\left( |(0,\alpha)|_s + |(0, \widetilde{\alpha})|_s \right) | \sigma \beta - \widetilde{\sigma} \widetilde{\beta}| dQ d\xi dH_s^0 dH_s^1 ds
 \\ &\leq \kappa (1+T) \int_{0}^{t}\int_{\R^{8d}}\left( |v-u| + |\widetilde{v} - \widetilde{u}| \right) | \sigma \beta - \widetilde{\sigma} \widetilde{\beta}| dH_s^0dH_s^1 ds.
 \end{align*}
 For $\e \in (0,\pi)$ let $\Xi_{\e} := [\e, \pi] \times S^{d-2}$ and $\Xi_{\e}^c = (0,\e) \times S^{d-2}$.
 Setting $\widetilde{\alpha}_0 := \alpha(\widetilde{v},\widetilde{u},\xi_0(v-u, \widetilde{v} - \widetilde{u},\xi))$ observe that, by Lemma \ref{PARAMETARIZATION} the function $\xi \longmapsto \xi_0(v-u,\widetilde{v} - \widetilde{u},\xi)$ has jacobian equal to $1$ so that we are allowed to insert it into the integral. This gives
 \begin{align*}
  J_1 &= \int_{0}^{t}\int_{\Xi \times \R^{8d}}\left( \sigma \beta \wedge  \widetilde{\sigma}\widetilde{\beta} \right) S(-s)\bigg\{ \psi(r,v+\alpha) - \psi(\widetilde{r}, \widetilde{v} + \widetilde{\alpha}_0) 
  \\ &\hskip60mm  - \psi(r,v) + \psi(\widetilde{r}, \widetilde{v})\bigg\}dQd\xi dH_s^0 dH_s^1 ds
  \\ &\leq (1 + T) \int_{0}^{t}\int_{\R^{8d}\times \Xi_{\e}^c} \left( \sigma \beta \wedge \widetilde{\sigma}\widetilde{\beta}\right) \left( |\alpha| + |\widetilde{\alpha}_0|\right)dQd\xi  dH_s^0dH_s^1ds
  \\ &+ \int_{0}^{t}\int_{\R^{8d}\times \Xi_{\e}} \left( \sigma \beta \wedge \widetilde{\sigma}\widetilde{\beta}\right)\bigg\{ |(r,v + \alpha) - (\widetilde{r},\widetilde{v} + \widetilde{\alpha}_0)|_s  
  \\ &\hskip50mm - |(r,v) - (\widetilde{r},\widetilde{v})|_s \bigg\} dQd\xi dH_s^0 dH_s^1 ds
  \\ &+ \int_{0}^{t}\int_{\R^{8d}\times \Xi_{\e}} \left( \sigma \beta \wedge \widetilde{\sigma}\widetilde{\beta}\right) \bigg\{ |(r,v) - (\widetilde{r},\widetilde{v})|_s 
  \\ &\hskip50mm- S(-s)\left( \psi(r,v) - \psi(\widetilde{r}, \widetilde{v})\right) \bigg\} dQd\xi dH_s^0 dH_s^1 ds
  \\ &=: J_1^{(1)} + J_1^{(2)} + J_1^{(3)}.
 \end{align*}
 Using $|(0,\alpha) - (0,\widetilde{\alpha})|_s \leq (1+T)|\alpha - \widetilde{\alpha}_0|$ and then \eqref{EQ:13} gives
 \begin{align*}
  J_1^{(2)}  &\leq \int_{0}^{t} \int_{\R^{8d} \times \Xi_{\e}} \left( \sigma \beta \wedge \widetilde{\sigma}\widetilde{\beta}\right) | (0,\alpha) - (0,\widetilde{\alpha}_0) |_s dQ d\xi dH_s^0 dH_s^1 ds
\\ &\leq (1+T) \int_{0}^{t} \int_{\R^{8d} \times \Xi_{\e}}\left( \sigma \beta \wedge \widetilde{\sigma}\widetilde{\beta}\right) | \alpha - \widetilde{\alpha}_0| dQd\xi dH_s^0 dH_s^1 ds
  \\ &\leq 2 \kappa (1+T) \int_{0}^{t}\int_{\R^{8d}} \left( \sigma \beta \wedge \widetilde{\sigma}\widetilde{\beta}\right) \left( |v - \widetilde{v}| + |u - \widetilde{u}| \right) dH_s^0 dH_s^1 ds.
 \end{align*}
  Setting $c_{\e} := \int_{\Xi_{\e}^c}\theta dQ(\theta)d\xi$, we obtain from the basic estimate $|\alpha(u,v,\theta,\xi)| \leq \theta|u-v|$ and $H_s^0, H_s^1 \in \mathcal{H}(\mu_s,\nu_s)$
 \begin{align*}
  J_1^{(1)} &\leq (1+T)c_{\e} \int_{0}^{t}\int_{\R^{8d}} \left( \sigma \beta \wedge \widetilde{\sigma}\widetilde{\beta}\right) \left( |v - u| + |\widetilde{v} - \widetilde{u}|\right) dH_s^0 dH_s^1 ds
  \\ &\leq (1+T)c_{\e} \int_{0}^{t}\int_{\R^{8d}} \sigma \beta ( |v| + |u| ) dH_s^0 dH_s^1 ds 
  \\ &\ \ \ + (1+T)c_{\e} \int_{0}^{t}\int_{\R^{8d}} \widetilde{\sigma} \widetilde{\beta} ( |\widetilde{v}| + |\widetilde{u}| ) dH_s^0 dH_s^1 ds
  \\ &= (1+T)c_{\e} \int_{0}^{t}\int_{\R^{4d}} \sigma \beta \left( |v| + |u| \right)\left( d\mu_s(q,u)d\mu_s(r,v) + d\nu_s(q,u)d\nu_s(r,v)\right)ds 
  \\ &=: c_{\e}h(t).
 \end{align*}
 For $J_1^{(3)}$ we break up the integrand into two parts, namely when $\sigma \beta < N$ and $\sigma \beta \geq N$ where $N \geq 1$.
 With $\kappa_{\e} := Q([\e,\pi))|S^{d-2}|$ we obtain
 \begin{align*}
  J_1^{(3)} &\leq N \kappa_{\e} \int_{0}^{t}\int_{\R^{8d}} \bigg\{ |(r,v) - (\widetilde{r},\widetilde{v})|_s - S(-s)\left( \psi(r,v) - \psi(\widetilde{r}, \widetilde{v})\right) \bigg\} dH_s^0 dH_s^1ds
  \\ &\ \ \ + \kappa_{\e} \int_{0}^{t}\int_{\R^{8d}} \1_{\{ \sigma \beta \geq N \} } \left( \sigma \beta \wedge \widetilde{\sigma}\widetilde{\beta}\right) \bigg\{ |(r,v) - (\widetilde{r},\widetilde{v})|_s 
  \\ & \qquad \qquad \qquad \qquad \qquad - S(-s)\left( \psi(r,v) - \psi(\widetilde{r}, \widetilde{v})\right) \bigg\} dH_s^0 dH_s^1ds
  \\ &\leq N \kappa_{\e} \int_{0}^{t} W_1^s(\mu_s, \nu_s)ds - \kappa_{\e}N \int_{0}^{t} \langle S(-s)\psi, \mu_s - \nu_s \rangle ds
  \\ &\ \ \ + 2\kappa_{\e} \int_{0}^{t}\int_{\R^{8d}} \1_{\{ \sigma \beta \geq N \} } \sigma \beta |(r,v) - (\widetilde{r},\widetilde{v})|_s dH_s^0 dH_s^1ds
 \end{align*}
 where we have used the fact that $\| S(-s)\psi \|_{s} \leq 1$,
 equation \eqref{FPEUNIQ:EQ02} and $H_s^0,H_s^1 \in \mathcal{H}(\mu_s,\nu_s)$. 
 For the last term we apply $|(r,v) - (\widetilde{r},\widetilde{v})|_s \leq (1+T)\left( |r| + |v| + |\widetilde{r}| + |\widetilde{v}|\right)$ and 
 then $H_s^0, H_s^1 \in \mathcal{H}(\mu_s,\nu_s)$ to obtain
 \begin{align*}
  J_1^{(3)} &\leq \kappa_{\e}N \int_{0}^{t}W_1^s(\mu_s, \nu_s) ds 
           - \kappa_{\e}N \int_{0}^{t} \langle S(-s)\psi, \mu_s - \nu_s \rangle ds + \kappa_{\e}g_N(t)
 \end{align*}
 where 
 \[
  g_N(t) = 2(1+T)\int_{0}^{t} \int_{\R^{4d}} \sigma \beta \1_{ \{ \sigma \beta \geq N \} } \left( |r| + |v|\right)\left( d\mu_s(q,u)d\mu_s(r,v) + d\nu_s(q,u)d\nu_s(r,v)\right)ds.
 \]
 Collecting all inequalities gives
 \begin{align*}
  \langle S(-t)\psi, \mu_t - \nu_t \rangle &\leq W_1(\mu_0, \nu_0) + 2 \kappa(1+T) \int_{0}^{t}\int_{\R^{8d}}\Psi dH_s^0 dH_s^1 ds
  \\ &\ \ \ + \int_{0}^{t}(\kappa_{\e}g_N(s) + c_{\e}h(s))ds
  + N \kappa_{\e} \int_{0}^{t}W_1^s(\mu_s,\nu_s)ds 
  \\ &\ \ \ - \kappa_{\e}N \int_{0}^{t}\langle S(-s)\psi, \mu_s - \nu_s \rangle ds.
 \end{align*}
 This yields 
 \begin{align*}
  \langle S(-t)\psi, \mu_t - \nu_t \rangle e^{\kappa_{\e}N t} &\leq W_1(\mu_0,\nu_0) + 2\kappa(1+T)\int_{0}^{t}e^{\kappa_{\e}Ns}\int_{\R^{8d}}\Psi dH_s^0 dH_s^1 ds
  \\ &\ \ \ + \int_{0}^{t}e^{\kappa_{\e}Ns}(\kappa_{\e}g_N(s) + c_{\e}h(s))ds + N \kappa_{\e} \int_{0}^{t}e^{\kappa_{\e}N s}W_1^s(\mu_s,\nu_s)ds.
 \end{align*}
 Taking the supremum over $\psi \in \mathrm{Lip}(\R^{2d})$ with $\| \psi \|_0 \leq 1$, using \eqref{COUP:00} and,
 then applying the generalized Gronwall inequality gives
 \begin{align*}
  W_1^t(\mu_t,\nu_t) &\leq W_1(\mu_0,\nu_0) + 2\kappa(1+T) \int_{0}^{t}\int_{\R^{8d}}\Psi dH_s^0 dH_s^1 ds
  \\ &\hskip30mm +  \int_{0}^{t}\left( \kappa_{\e}g_N(s) + c_{\e}h(s)\right)ds.
 \end{align*}
 Taking first $N \to \infty$ and then $\e \to 0$ proves the assertion.
\end{proof}

\section{Proof of Theorem \ref{FPE:UNIQTH:00}}
Assume that $\gamma \in [0,2]$ and let $\delta$ be given by \eqref{EQ:07}. Let us first show that under the conditions of Theorem \ref{FPE:UNIQTH:00}, also \eqref{FPE:12} is satisfied. 
Indeed, using $\sigma(|v-u|) \leq C \langle v \rangle^2 \langle u \rangle^2$ this follows from
\begin{align*}
 ( |v| + |u| + |r| + |q|)\sigma(|v-u|)
 &\leq C(|v| + |u|) \langle v \rangle^2 \langle u \rangle^2 + C(|r| + |q|)\langle v \rangle^2 \langle u \rangle^2
 \\ &\leq C \langle v \rangle^3 \langle u \rangle^3 + C \langle r \rangle \langle q \rangle \langle v \rangle \langle u \rangle
 \\ &\leq C \langle v \rangle^3 \langle u \rangle^3 
 + C \left( \langle r \rangle^{1 + \delta} + \langle v \rangle^{2 + \frac{2}{\delta}}\right)\left( \langle q \rangle^{1 + \delta} + \langle u \rangle^{2 + \frac{2}{\delta}}\right),
\end{align*}
where we have used Jensen's inequality
\[
 \langle r \rangle \langle v \rangle^2 \leq \frac{1}{1 + \delta}\langle r \rangle^{1 + \delta} + \frac{\delta}{1 + \delta}\langle v \rangle^{2\frac{1 + \delta}{\delta}}.
\]
Thus we can apply the coupling inequality. 
In order to estimate $\Psi$ as defined in Proposition \ref{general coupling inequality}, we need the following lemma.
\begin{Lemma}\label{FPE:LEMMA03}
 There exists a constant $C = C(\gamma, \sigma) > 0$ such that the following holds:
 \begin{enumerate}
  \item[(a)] For all $v,u, \widetilde{v}, \widetilde{u} \in \R^d$
   \begin{align*}
     &\ \left( |v-u| + |\widetilde{v} - \widetilde{u}| \right) \left| \sigma(|v-u|) - \sigma(|\widetilde{v} - \widetilde{u}|) \right| 
    \\ & \qquad \qquad \leq C \left( |v - u|^{\gamma} + |\widetilde{v} - \widetilde{u}|^{\gamma}\right)\left( |v-\widetilde{v}| + |u - \widetilde{u}| \right).
   \end{align*}
   Note that for $\gamma = 0$ the left-hand side is zero and hence the assertion is trivially satisfied.
  \item[(b)] For all $v,u,\widetilde{v}, \widetilde{u}$
   \begin{align*}
   &\ \left(|v-u| + |\widetilde{v} - \widetilde{u}|\right) | \sigma(|v-u|)\beta(r-q) - \sigma(|\widetilde{v} - \widetilde{u}|)\beta(\widetilde{r} - \widetilde{q})| 
    \\ &\leq C \left( \langle v \rangle^{\gamma} + \langle u \rangle^{\gamma} + \langle \widetilde{v} \rangle^{\gamma} + \langle \widetilde{u} \rangle^{\gamma}\right) \left( |v - \widetilde{v}| + |u - \widetilde{u}| \right)
    \\ &\ \ \ + C \left( \langle v \rangle^{1 + \gamma} + \langle u \rangle^{1 + \gamma} + \langle \widetilde{v} \rangle^{1+ \gamma} + \langle \widetilde{u} \rangle^{1 + \gamma}\right)\left( |r - \widetilde{r}| + |q-\widetilde{q}|\right).
   \end{align*}
 \end{enumerate}
\end{Lemma}
\begin{proof}
  \textit{(a)} We use, for $x,y \geq 0$ and $a,b > 0$, the elementary inequality
 \begin{align*}
  c_{a,b}|x^{a+b} - y^{a+b}| \leq (x^a + y^a)|x^b -y^b| \leq C_{a,b} |x^{a+b} - y^{a+b}|
 \end{align*}
 with some constants $c_{a,b},C_{a,b} > 0$ (see e.g. \cite{F15} for a similar application) to obtain
 \begin{align*}
  &\ \left( |v-u| + |\widetilde{v} - \widetilde{u}| \right) \left| \sigma(|v-u|) - \sigma(|\widetilde{v} - \widetilde{u}|) \right|
  \\ &\leq C\left( |v-u| + |\widetilde{v} - \widetilde{u}| \right) \left| |v-u|^{\gamma} - |\widetilde{v} - \widetilde{u}|^{\gamma} \right|
  \\ &\leq C\left| |v-u|^{1+\gamma} - |\widetilde{v} - \widetilde{u}|^{1+\gamma}\right|
      \leq C \left( |v-u|^{\gamma} + |\widetilde{v} - \widetilde{u}|^{\gamma}\right)\left( | v - \widetilde{v}| + |u - \widetilde{u}| \right).
 \end{align*}
 \\ \textit{(b)} We estimate by (a)
 \begin{align*}
   &\ |v-u| | \sigma(|v-u|)\beta(r-q) - \sigma(|\widetilde{v} - \widetilde{u}|)\beta(\widetilde{r} - \widetilde{q})| 
   \\ &\leq |v-u| \left( \beta(\widetilde{r} - \widetilde{q})|\sigma(|v-u|) - \sigma(|\widetilde{v} - \widetilde{u}|)| + \sigma(|v-u|)| \beta(r-q) - \beta(\widetilde{r} - \widetilde{q})| \right)
   \\ &\leq C \left( \langle v \rangle^{\gamma} + \langle u \rangle^{\gamma} + \langle \widetilde{v} \rangle^{\gamma} + \langle \widetilde{u} \rangle^{\gamma}\right) \left( |v-\widetilde{v}| + |u-\widetilde{u}|\right)
   + C \left( \langle v \rangle^{1+\gamma} + \langle u \rangle^{1+\gamma} \right)\left( |r-\widetilde{r}| + |q - \widetilde{q}| \right).
 \end{align*}
 The second term can be estimated in the same way. 
\end{proof}
From this lemma we deduce the following estimate on the function $\Psi$.
\begin{Lemma}\label{FPE:LEMMA05}
 There exists a constant $C = C(\delta,\gamma,\sigma) > 0$ such that
 \begin{align*}
   \Psi &\leq\ \  C\left( e^{\delta \langle u \rangle^{1+\gamma}} + e^{\delta \langle \widetilde{u} \rangle^{1+\gamma}} \right)\left(|v-\widetilde{v}| + |r-\widetilde{r}|\right)
  \\ &\ \ \ + C\left( e^{\delta \langle v \rangle^{1+\gamma}} + e^{\delta \langle \widetilde{v} \rangle^{1+\gamma} } \right)\left( |u - \widetilde{u}| + |q - \widetilde{q}|\right)
 \\ &\ \ \ + C\left( \langle v \rangle^{1+\gamma} + \langle \widetilde{v} \rangle^{1+\gamma} \right)\left( |v-\widetilde{v}| + |r - \widetilde{r}| \right)
 \\ &\ \ \ + C\left( \langle u \rangle^{1+\gamma} + \langle \widetilde{u} \rangle^{1+\gamma} \right)\left( |u - \widetilde{u}| + |q - \widetilde{q}| \right).
 \end{align*}
\end{Lemma}
\begin{proof}
Using the inequality
\begin{align*}
 &\  \sigma(|v-u|) \wedge \sigma(|\widetilde{v} - \widetilde{u}|) \left(| v- \widetilde{v}| + |u - \widetilde{u}|\right)
 \\ &\leq C \left(\langle u \rangle^{\gamma} + \langle v \rangle^{\gamma} + \langle \widetilde{u} \rangle^{\gamma} + \langle \widetilde{v} \rangle^{\gamma} \right) \left(|v-\widetilde{v}| + |u-\widetilde{u}|\right)
\end{align*}
we obtain from Lemma \ref{FPE:LEMMA03}
\begin{align*}
 \Psi &\leq C\left( \langle u \rangle^{\gamma} + \langle v \rangle^{\gamma} + \langle \widetilde{u} \rangle^{\gamma} + \langle \widetilde{v} \rangle^{\gamma} \right)\left( |v-\widetilde{v}| + |u - \widetilde{u}| \right)
 \\ &\ \ \ + C\left( \langle u \rangle^{1+\gamma} + \langle v \rangle^{1+\gamma} + \langle \widetilde{u} \rangle^{1+\gamma} + \langle \widetilde{v} \rangle^{1+\gamma} \right)\left( |r - \widetilde{r}| + |q-\widetilde{q}| \right)
 \\ &\leq  C\left( \langle u \rangle^{1+\gamma} + \langle \widetilde{u} \rangle^{1+\gamma} \right)\left(|v-\widetilde{v}| + |r-\widetilde{r}|\right)
 \\ &\ \ \ + C\left( \langle v \rangle^{1+\gamma} + \langle \widetilde{v} \rangle^{1+\gamma} \right)\left( |u - \widetilde{u}| + |q - \widetilde{q}|\right)
 \\ &\ \ \ + C\left( \langle v \rangle^{1+\gamma} + \langle \widetilde{v} \rangle^{1+\gamma} \right)\left( |v-\widetilde{v}| + |r - \widetilde{r}| \right)
 \\ &\ \ \ + C\left( \langle u \rangle^{1+\gamma} + \langle \widetilde{u} \rangle^{1+\gamma} \right)\left( |u - \widetilde{u}| + |q - \widetilde{q}| \right),
\end{align*}
where $C = C(\gamma,\sigma) > 0$ is some constant. 
Estimating the polynomials by the exponential function yields the assertion.
\end{proof}
The coupling inequality combined with Lemma \ref{FPE:LEMMA05} gives for some constant $K > 0$
\begin{align*}
  W_1^t( \mu_t,\nu_t) &\leq W_1(\mu_0, \nu_0) + K \mathcal{C}_{\gamma}(T,\mu+\nu,\delta) \int_{0}^{t} W_1^s(\mu_s,\nu_s)ds
  \\ &\ \ \ + K \mathcal{C}_{\gamma}(T,\mu+\nu,\delta)\int_{0}^{t}\int_{\R^{2d}}\left( \langle v \rangle^{1+\gamma} + \langle \widetilde{v} \rangle^{1+\gamma}\right)\left( |v- \widetilde{v}| + |r- \widetilde{r}|\right)dH_s ds,
\end{align*}
where we have used $|v - \widetilde{v}| + |r - \widetilde{r}| \leq (1+T)|(r,v) - (\widetilde{r}, \widetilde{v})|_s$,
\eqref{FPEUNIQ:EQ02} and $H_s = H_s(dr,dv;d\widetilde{r},d\widetilde{v}) \in \mathcal{H}(\mu_s,\nu_s)$. 
The proof of Theorem \ref{FPE:UNIQTH:00} is complete once we have shown the following Lemma.
\begin{Lemma}
 There exists a constant $K > 0$ such that for each $s \in [0,t]$ we have
 \begin{align}\label{ESTIMATE1}
  &\ \int_{\R^{4d}}\left( \langle v \rangle^{1+\gamma} + \langle \widetilde{v}\rangle^{1+\gamma}\right)|r - \widetilde{r}| dH_s(r,v;\widetilde{r},\widetilde{v})
  \\ \notag &\leq K\mathcal{C}_{\gamma}(T,\mu+\nu,\delta) \left( 1 + \left| \ln(W_1^s(\mu_s,\nu_s))\right| \right) W_1^s(\mu_s,\nu_s)
 \end{align}
 and
 \begin{align}\label{ESTIMATE}
    &\ \int_{\R^{4d}}\left( \langle v \rangle^{1+\gamma} + \langle \widetilde{v}\rangle^{1+\gamma}\right)|v - \widetilde{v}| dH_s(r,v;\widetilde{r},\widetilde{v})
   \\ \notag &\leq K\mathcal{C}_{\gamma}(T,\mu+\nu,\delta)\left( 1 + \left| \ln(W_1^s(\mu_s,\nu_s))\right| \right) W_1^s(\mu_s,\nu_s).
 \end{align} 
\end{Lemma}
A similar estimate to \eqref{ESTIMATE} was shown in \cite[p.820]{FM09},
while \eqref{ESTIMATE1} did not appear in the space-homogeneous setting studied there.
\begin{proof}
 Here and below we denote by $K > 0$ some generic constant (independent of $(\mu_t)_{t \in [0,T]}$ and $(\nu_t)_{t \in [0,T]}$) which may vary from line to line. For $b > 0$, we obtain
 \begin{align*}
  &\ \int_{\R^{4d}}\left( \langle v \rangle^{1+\gamma} + \langle \widetilde{v}\rangle^{1+\gamma}\right)|r - \widetilde{r}| dH_s(r,v;\widetilde{r},\widetilde{v})
 \\ &= \int_{\R^{4d}}\1_{ \{ \langle v \rangle \leq b, \ \langle \widetilde{v} \rangle \leq b \} }\left( \langle v \rangle^{1+\gamma} + \langle \widetilde{v}\rangle^{1+\gamma}\right)|r - \widetilde{r}| dH_s(r,v;\widetilde{r},\widetilde{v})
 \\ &\ \ \  + \int_{\R^{4d}}\left(1 - \1_{ \{ \langle v \rangle \leq b, \ \langle \widetilde{v} \rangle \leq b \} } \right) \left( \langle v \rangle^{1+\gamma} + \langle \widetilde{v}\rangle^{1+\gamma}\right)|r - \widetilde{r}| dH_s(r,v;\widetilde{r},\widetilde{v})
  \\ &\leq 2 b^{1+\gamma}(1+T) W_1^s(\mu_s, \nu_s) + 2I(b)
 \end{align*}
 where we have used $|r - \widetilde{r}| \leq (1+T)|(r,v) - (\widetilde{r},\widetilde{v})|_s$ and
 \begin{align*}
  I(b) := \int_{\R^{4d}}\left( \langle v \rangle^{1+\gamma} + \langle \widetilde{v} \rangle^{1+\gamma}\right)(|r| + |\widetilde{r}|) 
  \bigg( \1_{ \{ \langle v \rangle > b \} } + \1_{ \{ \langle \widetilde{v} \rangle > b \} } \bigg)dH_s^1.
 \end{align*}
 For $K$ large enough satisfying for all $v, \widetilde{v} \in \R^d$
 \begin{align*}
  \left( \langle v \rangle^{1+\gamma} + \langle \widetilde{v}\rangle^{1+\gamma}\right)\left( e^{\frac{\delta}{1+\delta}\frac{\delta}{2}\langle v \rangle^{1+\gamma}} + e^{\frac{\delta}{1+\delta}\frac{\delta}{2}\langle \widetilde{v}\rangle^{1+\gamma}}\right)
  \leq K\left( e^{\frac{\delta}{1+\delta}\delta\langle v \rangle^{1+\gamma}} + e^{\frac{\delta}{1+\delta}\delta\langle \widetilde{v}\rangle^{1+\gamma}}\right)
 \end{align*}
 we obtain 
 \[
  \left( \langle v \rangle^{1+\gamma} + \langle \widetilde{v} \rangle^{1+\gamma}\right)\1_{ \{ \langle v \rangle > b \} }
 \leq K e^{-\frac{\delta}{1+\delta} \frac{\delta}{2}b^{1+\gamma}} \left( e^{\frac{\delta}{1+\delta}\delta\langle v \rangle^{1+\gamma}} + e^{\frac{\delta}{1+\delta}\delta\langle \widetilde{v}\rangle^{1+\gamma}}\right).
 \]
 Estimating the second term in $I(b)$ in the same way gives
 \begin{align*}
  I(b) &\leq K e^{-\frac{\delta}{1+\delta}\frac{\delta}{2}b^{1+\gamma}} \int_{\R^{4d}}\left( e^{\frac{\delta}{1+\delta}\delta\langle v \rangle^{1+\gamma}} + e^{\frac{\delta}{1+\delta}\delta\langle \widetilde{v}\rangle^{1+\gamma}}\right)(|r| + |\widetilde{r}|)dH_s(r,v;\widetilde{r},\widetilde{v})
  \\ &\leq K e^{-\frac{\delta}{1+\delta}\frac{\delta}{2}b^{1+\gamma}} \mathcal{C}_{\gamma}(T, \mu+\nu,\delta),
 \end{align*}
 where we have used $H_s \in \mathcal{H}(\mu_s,\nu_s)$ and
 \begin{align*}
  &\ \left( e^{\frac{\delta}{1+\delta}\delta\langle v \rangle^{1+\gamma}} + e^{\frac{\delta}{1+\delta}\delta\langle \widetilde{v}\rangle^{1+\gamma}}\right)(|r| + |\widetilde{r}|)
  \\ &\leq \frac{\delta}{1+\delta} \left( e^{\frac{\delta}{1+\delta}\delta\langle v \rangle^{1+\gamma}} + e^{\frac{\delta}{1+\delta}\delta\langle \widetilde{v}\rangle^{1+\gamma}}\right)^{\frac{1+\delta}{\delta}} + \frac{1}{1+\delta}\left( |r| + |\widetilde{r}|\right)^{1+\delta}
  \\ &\leq K \left( e^{\frac{\delta}{2} \langle v \rangle^{1+\gamma}} + e^{\frac{\delta}{2}\langle \widetilde{v}\rangle^{1+\gamma}}\right)
  + K \left( |r|^{1 + \delta} + |\widetilde{r}|^{1+\delta} \right).
 \end{align*}
 This gives the estimate
 \begin{align*}
  &\ \int_{\R^{4d}}\left( \langle v \rangle^{1+\gamma} + \langle \widetilde{v}\rangle^{1+\gamma}\right)|r - \widetilde{r}| dH_s(r,v;\widetilde{r},\widetilde{v})
  \\ &\leq 2 b^{1+\gamma}(1+T) W_1^s(\mu_s, \nu_s) + K e^{-\frac{\delta}{1+\delta}\frac{\delta}{2}b^{1+\gamma}} \mathcal{C}_{\gamma}(T, \mu+\nu,\delta).
 \end{align*}
 Letting $b^{1+\gamma} = \frac{1+\delta}{\delta}\frac{2}{\delta}\left| \ln\left( W_1^s(\mu_s,\nu_s)\right)\right|$ yields \eqref{ESTIMATE1}.
 The inequality \eqref{ESTIMATE} can be shown in the same way.
\end{proof}

\section{Proof of Theorem \ref{FPE:UNIQTH:01}}
In this section we assume that $\gamma \in (-d,0]$.
As before, we first check that \eqref{FPE:12} is satisfied.
Indeed we obtain
\begin{align*}
 &\ \int_{\R^{2d}}\int_{\R^{2d}}( |v| + |u| + |r| + |q|)\sigma(|v-u|) \mu_t(dq,du) \mu_t(dr,dv)
 \\ &\leq K \int_{\R^{2d}}\int_{\R^{2d}}( |v| + |u| + |r| + |q| )|v-u|^{\gamma} \mu_t(dq,du) \mu_t(dr,dv)
 \\ &\leq K \Lambda(\mu_t) \int_{\R^{2d}}(|v| + |r|)\mu_t(dr,dv)
\end{align*}
for some generic constant $K > 0$. Since an analogous 
inequality also holds for $\nu_t$, \eqref{FPE:12} is satisfied
and we may apply the coupling inequality. 
Again, to estimate the expression $\Psi$ appearing in Proposition \ref{general coupling inequality} we use the following lemma.
\begin{Lemma}\label{FPE:LEMMA04}
 There exists a constant $C = C(\gamma,\sigma) > 0$ such that the following holds:
 \begin{enumerate}
  \item[(a)] For all $v,u,\widetilde{v},\widetilde{u} \in \R^d$
   \begin{align}\label{UNIQ:07}
     &\ \left( |v-u| + |\widetilde{v} - \widetilde{u}| \right) \left| \sigma(|v-u|) - \sigma(|\widetilde{v} - \widetilde{u}|) \right| 
    \\ \notag & \qquad \qquad \qquad \leq C \left( |v - u|^{\gamma} + |\widetilde{v} - \widetilde{u}|^{\gamma}\right)\left( |v- \widetilde{v}| + |u - \widetilde{u}| \right).
   \end{align}
  \item[(b)] If $\gamma \in (-1,0]$, then for all $v,u,\widetilde{v},\widetilde{u} \in \R^d$
  \begin{align*}
    &\ \left( |v-u| + |\widetilde{v} - \widetilde{u}|\right) | \sigma(|v-u|)\beta(r-q) - \sigma(|\widetilde{v} - \widetilde{u}|)\beta(\widetilde{r} - \widetilde{q})| 
    \\ &\leq C \left( \langle v \rangle^{1+\gamma} + \langle u \rangle^{1+\gamma} + \langle \widetilde{v} \rangle^{1+\gamma} + \langle \widetilde{u} \rangle^{1+\gamma}\right)\left( |r - \widetilde{r}| + |q - \widetilde{q}| \right)
    \\ &\ \ \  + C \left( |v - u|^{\gamma} + |\widetilde{v} - \widetilde{u}|^{\gamma}\right)\left( |v-\widetilde{v}| + |u - \widetilde{u}| \right).
  \end{align*}
  \item[(c)] If $\gamma \in (-d,-1]$, then for all $v,u,\widetilde{v}, \widetilde{u} \in \R^d$
  \begin{align*}
    &\ \left( |v-u| + |\widetilde{v} - \widetilde{u}|\right) | \sigma(|v-u|)\beta(r-q) - \sigma(|\widetilde{v} - \widetilde{u}|)\beta(\widetilde{r} - \widetilde{q})| 
    \\ &\leq C \left( |v-u|^{1+\gamma} + |\widetilde{v} - \widetilde{u}|^{1+ \gamma}\right)\left( |r - \widetilde{r}| + |q - \widetilde{q}| \right)
    \\ &\ \ \ + C \left( |v - u|^{\gamma} + |\widetilde{v} - \widetilde{u}|^{\gamma}\right)\left( |v-\widetilde{v}| + |u - \widetilde{u}| \right).
  \end{align*}
 \end{enumerate}
\end{Lemma}
\begin{proof}
 \textit{(a)} Following \cite[p. 821]{FM09} we obtain
 \begin{align*}
  &\ |v-u| \left| \sigma(|v-u|) - \sigma(|\widetilde{v} - \widetilde{u}|) \right|
  \\ &\leq c_{\sigma} \left( |v-u| \wedge |\widetilde{v} - \widetilde{u}| + \left| |v-u| - |\widetilde{v} - \widetilde{u}| \right| \right) \left| |v-u|^{\gamma} - |\widetilde{v} - \widetilde{u}|^{\gamma} \right|
  \\ &\leq  C|\gamma| \left( |v-u| \wedge |\widetilde{v} - \widetilde{u}|\right)^{\gamma - 1}  \left| |v-u|^{\gamma} - |\widetilde{v} - \widetilde{u}|^{\gamma} \right| |v-u| \wedge |\widetilde{v} - \widetilde{u}|
  \\ &\ \ \ + C \left( |v-u|^{\gamma} \vee |\widetilde{v} - \widetilde{u}|^{\gamma} \right) \left| |v-u|^{\gamma} - |\widetilde{v} - \widetilde{u}|^{\gamma} \right|
  \\ &\leq C (1 + |\gamma|) \left( |v-u|^{\gamma} + |\widetilde{v} - \widetilde{u}|^{\gamma}\right)\left( | v - \widetilde{v}| + |u - \widetilde{u}| \right).
 \end{align*}
 The other term can be estimated in the same way.
 \\ \textit{(b)} By (a) it follows that
 \begin{align*}
   &\ |v-u| | \sigma(|v-u|)\beta(r-q) - \sigma(|\widetilde{v} - \widetilde{u}|)\beta(\widetilde{r} - \widetilde{q})| 
   \\ &\leq |v-u|\left( \sigma(|v-u|)|\beta(r-q) - \beta(\widetilde{r} - \widetilde{q})| + \beta(\widetilde{r} - \widetilde{q})|\sigma(|v-u|) - \sigma(|\widetilde{v} - \widetilde{u}|)| \right)
   \\ &\leq C \left( \langle v \rangle^{1+\gamma} + \langle u \rangle^{1+\gamma}\right)\left( | r- \widetilde{r}| + |q - \widetilde{q}|\right) 
   + C \left( |v - u|^{\gamma} + |\widetilde{v} - \widetilde{u}|^{\gamma}\right)\left( |v-\widetilde{v}| + |u - \widetilde{u}| \right).
 \end{align*}
 The other term is estimated in the same way.
 \\ \textit{(c)} In this case we obtain
 \begin{align*}
 &\ |v-u| | \sigma(|v-u|)\beta(r-q) - \sigma(|\widetilde{v} - \widetilde{u}|)\beta(\widetilde{r} - \widetilde{q})| 
   \\ &\leq |v-u|\left( \sigma(|v-u|)|\beta(r-q) - \beta(\widetilde{r} - \widetilde{q})| + \beta(\widetilde{r} - \widetilde{q})|\sigma(|v-u|) - \sigma(|\widetilde{v} - \widetilde{u}|)| \right)
   \\ &\leq C |v-u|^{1 + \gamma}\left( |r - \widetilde{r}| + |q - \widetilde{q}|\right) + C\left( |v - u|^{\gamma} + |\widetilde{v} - \widetilde{u}|^{\gamma}\right)\left( |v-\widetilde{v}| + |u - \widetilde{u}| \right)
 \end{align*}
 and similarly for the second term.
\end{proof}
From this we deduce the following estimates for $\Psi$.
\begin{Lemma}\label{FPE:LEMMA06}
 There exists a constant $C > 0$ such that
 \begin{enumerate}
  \item[(a)] If $\gamma \in (-d,-1]$, then
 \begin{align*}
  \Psi &\leq C\left( |v-u|^{1+\gamma} + |\widetilde{v} - \widetilde{u}|^{1+\gamma}\right)\left(|r-\widetilde{r}| + |q - \widetilde{q}|\right)
  \\ &\ \ \ + C \left( |v-u|^{\gamma} + |\widetilde{v} - \widetilde{u}|^{\gamma}\right)\left( |v-\widetilde{v}| + |u - \widetilde{u}|\right).
 \end{align*}
 \item[(b)] If $\gamma \in (-1,0)$, then
 \begin{align*}
  \Psi &\leq C\left( \langle v \rangle^{1+ \gamma} + \langle u \rangle^{1+ \gamma} + \langle \widetilde{v}\rangle^{1+\gamma} + \langle \widetilde{u}\rangle^{1+\gamma}\right)\left(|r-\widetilde{r}| + |q - \widetilde{q}|\right)
  \\ &\ \ \ + C \left( |v-u|^{\gamma} + |\widetilde{v} - \widetilde{u}|^{\gamma}\right)\left( |v-\widetilde{v}| + |u - \widetilde{u}|\right)
 \end{align*}
 \end{enumerate}
\end{Lemma}
\begin{proof}
Lemma \ref{FPE:LEMMA06} is now a consequence of 
\begin{align*}
 &\ \sigma(|v-u|) \wedge \sigma(|\widetilde{v} - \widetilde{u}|) \left( |v- \widetilde{v}| + |u - \widetilde{u}|\right)
 \\ &\qquad \qquad \qquad \leq C \left( |v-u|^{\gamma} + |\widetilde{v} - \widetilde{u}|^{\gamma} \right)\left( |v-\widetilde{v}| + |u - \widetilde{u}|\right)
\end{align*}
and Lemma \ref{FPE:LEMMA04} from the appendix.
\end{proof}
Below we treat the cases $\gamma \in (-d,-1]$ and $\gamma \in (-1,0)$ seperately.

\subsection{Case $\gamma \in (-d,-1]$}
 We obtain from the general coupling inequality together with Lemma \ref{FPE:LEMMA06} and $H_s \in \mathcal{H}(\mu_s,\nu_s)$
 \begin{align*}
  & \ \ \ W_1^t( \mu_t, \nu_t) \leq W_1(\mu_0, \nu_0) 
  \\ &\ \ \  + K\int_{0}^{t}\int_{\R^{8d}} \left( |v-u|^{\gamma} + |\widetilde{v} - \widetilde{u}|^{\gamma}\right)\left( |v-\widetilde{v}| + |u - \widetilde{u}|\right) dH_s(q,u;\widetilde{q},\widetilde{u}) dH_s(r,v;\widetilde{r},\widetilde{v}) ds
  \\ &\ \ \ + K\int_{0}^{t}\int_{\R^{8d}}\left( |v-u|^{1 + \gamma} + |\widetilde{v} - \widetilde{u}|^{1 + \gamma}\right)\left( |r-\widetilde{r}| + |q - \widetilde{q}|\right)dH_s(q,u;\widetilde{q},\widetilde{u})  dH_s(r,v;\widetilde{r},\widetilde{v})ds
 \\ &\leq W_1(\mu_0,\nu_0) 
 +  K \int_{0}^{t}\int_{\R^{4d}} \Lambda(s,\mu,\nu) |v - \widetilde{v}| dH_s(r,v;\widetilde{r},\widetilde{v})
 \\ &\ \ \ + K \int_{0}^{t}\int_{\R^{4d}} \widetilde{\Lambda}(s,\mu,\nu)|r-\widetilde{r}|dH_s(r,v;\widetilde{r},\widetilde{v}) 
  \\ &\leq W_1(\mu_0,\nu_0) + K \int_{0}^{t}\left( \Lambda(s,\mu,\nu) + \widetilde{\Lambda}(s, \mu,\nu)\right)W_1^s(\mu_s,\nu_s)ds,
 \end{align*}
 where we have used $|v - \widetilde{v}| \leq |(r,v) - (\widetilde{r}, \widetilde{v})|_s$ and 
 $|r - \widetilde{r}| \leq (1+T) |(r,v) - (\widetilde{r}, \widetilde{v})|_s$ and
 \[
  \widetilde{\Lambda}(s,\mu,\nu) = \sup \limits_{u \in \R^d}\int_{\R^{2d}} |v-u|^{1+\gamma}d(\mu_s + \nu_s)(r,v).
 \]
Since $1 + \gamma \leq 0$ we obtain  $|v-u|^{1+\gamma} \leq 1 + |v-u|^{\gamma}$ and hence
 $\widetilde{\Lambda}(s,\mu,\nu) \leq 1 + \Lambda(s,\mu,\nu)$ gives
\[
 W_1^t( \mu_t, \nu_t) \leq W_1(\mu_0, \nu_0) + K \int_{0}^{t}\left( 1 + \Lambda(s,\mu,\nu)\right) W_1^s(\mu_s,\nu_s)ds.
\]
The assertion follows from the classical Gronwall lemma.
 
\subsection{Case $\gamma \in (-1,0)$}
Proceeding as before we obtain from the coupling inequality and Lemma \ref{FPE:LEMMA06}
 \begin{align*}
  &\ W_1^t( \mu_t, \nu_t) \leq W_1(\mu_0, \nu_0) 
  \\ &\ \ \ + K\int_{0}^{t}\int_{\R^{8d}} \left( |v-u|^{\gamma} + |\widetilde{v} - \widetilde{u}|^{\gamma}\right)\left( |v-\widetilde{v}| + |u - \widetilde{u}|\right) dH_s(q,u;\widetilde{q},\widetilde{u})  dH_s(r,v;\widetilde{r},\widetilde{v}) ds
  \\ &\ \ \ + K\int_{0}^{t}\int_{\R^{8d}}\left( \langle u \rangle^{1+\gamma} + \langle \widetilde{u} \rangle^{1+\gamma}\right)|r-\widetilde{r}|  dH_s(q,u;\widetilde{q},\widetilde{u})  dH_s(r,v;\widetilde{r},\widetilde{v}) ds
  \\ &\ \ \ + K\int_{0}^{t}\int_{\R^{8d}}\left( \langle v \rangle^{1+\gamma} + \langle \widetilde{v} \rangle^{1+\gamma}\right)|r-\widetilde{r}|   dH_s(r,v;\widetilde{r},\widetilde{v}) ds
  \\ &\leq W_1(\mu_0, \nu_0) + K \int_{0}^{t} \left( 1 + \Lambda(s,\mu,\nu) \right) W_1^s(\mu_s,\nu_s)ds
  \\ &\ \ \ + K \int_{0}^{t}\int_{\R^{4d}}\left( \langle v \rangle^{1+\gamma} + \langle \widetilde{v} \rangle^{1+\gamma}\right)|r- \widetilde{r}| dH_s(r,v;\widetilde{r},\widetilde{v}) ds
  \\ &\leq W_1(\mu_0, \nu_0) + K \mathcal{C}_{\gamma}(T, \mu+\nu,\delta) \int_{0}^{t} W_1^s(\mu_s,\nu_s)(1 + \Lambda(s,\mu,\nu) + | \ln(W_1^s(\mu_s,\nu_s))|)ds,
 \end{align*}
 where in the last inequality we have used similar arguments as in the proof
 of \eqref{ESTIMATE1}.
 This implies the assertion.

\appendix

\section{Some Gronwall inequality}

The following lemma is due to \cite[Lemma 5.2.1, p. 89]{C95}.
\begin{Lemma}\label{UNIQ:LEMMA02}
 Let $\rho$ be a nonnegative bounded function on $[0,T]$, $a \in [0, \infty)$ and $g$ be a strictly positive and non-decreasing function on $(0,\infty)$.
 Suppose that $\int_{0}^{1}\frac{dx}{g(x)} = \infty$ and 
 \[
  \rho(t) \leq a + \int_{0}^{t}g(\rho(s))ds, \ \ t \in [0,T].
 \]
 Then
 \begin{enumerate}
  \item[(a)] If $a = 0$, then $\rho(t) = 0$ for all $t \in [0,T]$.
  \item[(b)] If $a > 0$, then $G(a) - G(\rho(t)) \leq t$ where $G(x) = \int_{x}^{1}\frac{dy}{g(y)}$.
 \end{enumerate}
\end{Lemma}

\bibliographystyle{amsplain}
\addcontentsline{toc}{section}{\refname}\bibliography{Bibliography}

\end{document}